\documentclass[sts]{imsart}

\usepackage{geometry}
\usepackage{amssymb,amsbsy,amsmath,amscd,amsthm,amsfonts,MnSymbol}
\usepackage{cite}
\usepackage[utf8]{inputenc}
\usepackage{enumerate,hyperref}
\usepackage{dsfont}
\usepackage{graphicx}

\newtheorem{definition}{Definition}
\newtheorem{theorem}{Theorem}

\newtheorem{lemma}{Lemma}
\newtheorem{remark}{Remark}

\newcommand{\R}{\mathbb R}

\newcommand{\E}{\mathbb E}

\newcommand{\N}{\mathbb N}
\newcommand{\Sp}{\mathbb S}

\newcommand{\Exp}{\textrm{Exp}}
\newcommand{\Log}{\textrm{Log}}

\newcommand{\Hess}{\textrm{H}}

\newcommand{\CAT}{\textrm{CAT}}

\newcommand{\per}{\mathrm{per}}

\newcommand{\var}{\mathrm{var}}

\newcommand*\diff{\mathop{}\!\mathrm{d}}

\newcommand{\DS}{\displaystyle}

\newenvironment{myproof}[2] {\paragraph{\textsc{Proof of {#1}}.}}{\hfill$\square$\newline}

\begin{document}

\begin{frontmatter}

\title{Geodesically convex $M$-estimation in metric spaces}
\runtitle{Convex $M$-estimation in metric spaces}

\author{Victor-Emmanuel Brunel \thanks{CREST-ENSAE, victor.emmanuel.brunel@ensae.fr}}

\runauthor{V.-E. Brunel}

\setattribute{abstractname}{skip} {{\bf Abstract:} } 
\begin{abstract}
We study the asymptotic properties of geodesically convex $M$-estimation on non-linear spaces. Namely, we prove that under very minimal assumptions besides geodesic convexity of the cost function, one can obtain consistency and asymptotic normality, which are fundamental properties in statistical inference. Our results extend the Euclidean theory of convex $M$-estimation; They also generalize limit theorems on non-linear spaces which, essentially, were only known for barycenters, allowing to consider robust alternatives that are defined through non-smooth $M$-estimation procedures.

\end{abstract}


\end{frontmatter}

\section{Preliminaries}

\subsection{Introduction}

The problem of $M$-estimation, or empirical risk minimization, is pervasive to statistics and machine learning. In many problems, one seeks to minimize a function that takes the form of an expectation, i.e., $\Phi(x)=\E[\phi(Z,x)]$, where $Z$ is a random variable with unknown distribution, $x$ is the target variable, which lives in some target space, and $\phi$ can be seen as a cost function. Classification, regression, location estimation, maximum likelihood estimation, are standard instances of such problems. In practice, the distribution of $Z$ is unknown, so the function $\Phi$ cannot be evaluated. Hence, one rather minimizes a proxy of the form $\hat\Phi_n(x)=n^{-1}\sum_{i=1}^n \phi(Z_i,x)$, where $Z_1,\ldots,Z_n$ are available i.i.d. copies of $Z$. This method is called $M$-estimation, or empirical risk minimization. When the target space is Euclidean, a lot of results are available for guarantees of a minimizer $\hat x_n$ of $\hat\Phi_n$, as an estimator for a minimizer $x^*$ of $\Phi$. Asymptotic results (such as consistency and asymptotic normality) are classic and well known, especially under technical smoothness assumptions on the cost function $\phi$ \cite{van1996weak}, most of which can be dropped when the cost function is convex in the target variable \cite{haberman1989concavity,niemiro1992asymptotics}.

In this work, we are interested in the framework in which the target space is not Euclidean, and does not even exhibit any linear structure. Of course, the space needs to be equipped with a metric, at the very least in order to measure the accuracy of an estimation procedure. Thus, we consider a target space that is a metric space, which we denote by $(M,d)$. 

Statistics and machine learning are more and more confronted with data that lie in non-linear spaces: In spatial statistics (e.g., directional data), computational tomography (e.g., data in quotient spaces such as in shape statistics, collected up to rigid transformations), economics (e.g., optimal transport, where data are measures), etc. Moreover, data and/or target parameters that are encoded as very high dimensional vectors may have a much smaller intrinsic dimension: For instance, if they are lying on small dimensional submanifolds of the ambient Euclidean space. In that case, leveraging the possibly non-linear geometry of the problems at hand can be a powerful tool in order to significantly reduce their dimensionality. This is the central motivating idea behind manifold learning \cite{genovese2012minimax,aamari2019nonasymptotic} or generative adversarial networks \cite{schreuder2021statistical}. Even though more and more algorithms are developed to work and, in particular, optimize in non-linear spaces \cite{LimPalfia14,ohtapalfia,zhang2016first,zhang2018towards,antonakopoulos2020online,criscitiello2022accelerated}, still little is understood from a statistical prospective.

A specific instance of $M$-estimation on metric spaces, which has attracted much attention, is that of barycenters. Given $n$ points $x_1,\ldots,x_n\in M$, a barycenter is any minimizer $x\in M$ of the sum of squared distances $\sum_{i=1}^n d(x,x_i)^2$. One can easily check that if $(M,d)$ is a Euclidean or Hilbert space, then the minimizer is unique and it is given by the average of $x_1,\ldots,x_n$. Barycenters extend the notion of average to spaces that lack a linear structure. More generally, given a probability measure $P$ on $M$, one can define a barycenter of $P$ as any minimizer $x\in M$ of $\int_M d(x,z)^2\diff P(z)$, provided that integral is finite for at least one (and hence, for all) value of $x\in M$. Barycenters were initially introduced in statistics by \cite{Frechet48} in the 1940's, and later by \cite{karcher1977Riemannian}, where they were rather called Fréchet means or Karcher means. They were popularized in the fields of shape statistics \cite{kendall2009shape}, optimal transport \cite{agueh2011barycenters,cuturi2014fast,le2017existence,claici2018stochastic,kroshnin2019complexity,kroshnin2021statistical,altschuler2021wasserstein,altschuler2022wasserstein}. and matrix analysis \cite{bhatia2006riemannian,bhatia2009positive,bhatia2019bures}. 
Asymptotic theory is fairly well understood for empirical barycenters in various setups, particularly laws of large numbers \cite{ziezold1977expected,sturm03} and central limit theorems in Riemannian manifolds (it is natural to impose a smooth structure on $M$ in order to derive central limit theorems) \cite{bhattacharya2003large,bhattacharya2005large,kendall2011limit,huckemann2011intrinsic,bhattacharya2017omnibus,eltzner2019smeary,eltzner2019stability}.
A few finite sample guarantees are available, even though the non-asymptotic statistical theory for barycenters is still quite limited \cite{sturm03,Funano10,schotz2019convergence,ahidar20,legouic22}. Asymptotic theorems are also available for more general $M$-estimators on Riemannian manifolds, e.g., $p$-means, where $\mathcal Z=X$ and $\phi(z,x)=d(x,z)^p$, for some $p\geq 1$ \cite{schotz2019convergence}, including geometric medians as a particular case ($p=1$), but only suboptimal rates are known. In particular, the standard $n^{-1/2}$-rate is unknown in general, beyond barycenters. It should be noted that all central limit theorems for barycenters in Riemannian manifolds rely on the smoothness of the cost function $\phi=d^2$ and use standard techniques from smooth $M$-estimation \cite{van1996weak} by performing Taylor expansions in local charts. Moreover, they assume that the data are almost surely in a ball with small radius, which essentially ensures the convexity of the squared distance function $d^2$ in its second variable, even though this synthetic property is not leveraged in this line of work. For instance, asymptotic normality cannot be obtained for geometric medians or other robust alternatives to barycenters, using these techniques.

%\TODO{mention M-estimation for robustness, cite Huber, etc. Cite the asymptotic results that are known: consistency in metric spaces, BP CLT, p-means, etc. Cite few non-asymptotic results that are known for barycenters. Then, mention Haberman and Niemiro for convex M-estimation. Then, a whole paragraph on the very hot research on stats and optim on metric spaces and Riemannian manifolds! Why important: blabla, uncertainty quantification, non-linear data, dimension reduction, etc.}

%The framework that we consider here is that of geodesic metric spaces, that is, metric spaces in which any two points can be connected by a shortest path. This is the most natural framework that allows to extend the notion of convexity, which is ubiquitous in both learning theory and optimization. That is, we assume that the cost function $\phi(\cdot,\cdot)$ is convex in its second argument, in a sense that we specify below. Convexity, as simple and elementary and as it is, is an extremely rich property that is known to yield powerful results, both for statistical and algorithmic guarantees, in linear spaces. In this work, our goal, essentially, is to convey that convexity can also be leveraged in non standard setups so as to obtain strong learning guarantees. 

\subsection{Contributions}

In this work, we prove strong consistency and asymptotic normality of geodesically convex $M$-estimators under very mild assumptions. We recover the central limit theorems proven for barycenters in Riemannian manifolds with small diameter, but we cover a much broader class of location estimators, including robust alternatives to barycenters. Just as in the Euclidean case, ours results convey that convexity is a powerful property, since it yields strong learning guarantees under very little assumptions. 

All our results are asymptotic, but it is worth mentioning that proving learning guarantees in non-linear spaces is a very challenging, ongoing research topic, because it requires non-standard tools from metric and differential geometry. In particular, very few non-asymptotic learning guarantees are available, even for simple estimators such as barycenters, and asymptotic theory is still an active research area in that framework (e.g., understanding non-standard asymptotic rates, a.k.a. sticky and smeary central limit theorems \cite{hotz2013sticky,eltzner2019smeary}). Moreover, asymptotic guarantees such as asymptotic normality provide benchmarks that are also important for a finite sample theory.

\section{Background and notation}

\subsection{Geodesic spaces} \label{sec:geodspaces}
	
Let $(M,d)$ be a complete and separable metric space that satisfies the following property: For all $x,y\in M$, there exists a continuous function $\gamma:[0,1]\to M$ such that $\gamma(0)=x$, $\gamma(1)=y$ and with the property that $d(\gamma(s),\gamma(t))=|s-t|d(x,y)$, for all $s,t\in[0,1]$. Such a function $\gamma$ is called a unit speed geodesic from $x$ to $y$ and it should be thought of as a (not necessarily unique) shortest path from $x$ to $y$. We denote by $\Gamma_{x,y}$ the collection of all such functions. The space $(M,d)$ is then called a geodesic space.
Examples of geodesic spaces include Euclidean spaces, Euclidean spheres, metric graphs, Wasserstein spaces, quotient spaces, etc. \cite{burago2001course,bridson2013metric}. Geodesic spaces allow for a natural extension of the notion of convexity.

\begin{definition}
	A function $f:M\to\R$ is called geodesically convex (convex, for short) if for all $x,y\in M, \gamma\in\Gamma_{x,y}$ and $t\in [0,1]$, $\DS f(\gamma(t))\leq (1-t)f(x)+tf(y)$.
\end{definition}

In this work, we will further assume that $(M,d)$ is a proper space, i.e., all bounded closed sets are compact. This assumption may seem limiting in practice but, at a high level, it essentially only discards infinite dimensional spaces. Indeed, Hopf-Rinow theorem \cite[Chapter 1, Proposition 3.7]{bridson2013metric} asserts that so long as $(M,d)$ is complete and locally compact, then it is proper. For instance, Hopf-Rinow theorem for Riemannian manifolds asserts that any complete (finite dimensional) Riemannian manifold is a proper geodesic metric space \cite[Chapter 7, Theorem 2.8]{doCarmo1992riemannian}.

\subsection{Riemannian manifolds}

A Riemannian manifold $(M,g)$ is a finite dimensional smooth manifold $M$ equipped with a Riemannian metric $g$, i.e., a smooth family of scalar products on tangent spaces. We refer the reader to \cite{doCarmo1992riemannian} and \cite{lee2012smooth,lee2018introduction} for a clear introduction to differential geometry and Riemannian manifolds. The distance inherited on $M$ from the Riemannian metric $g$ will be denoted by $d$. In this work, for simplicity, we only consider Riemannian manifolds without boundary, even though our results extend easily to the case with boundary.

The tangent bundle of $M$ is denoted by $TM$, i.e., $\DS TM=\bigcup_{x\in M} T_xM$ where, for all $x\in M$, $T_xM$ the tangent space to $M$ at $x$. For all $x\in M$, we also let $\Exp_x:T_xM\to M$ and $\Log_x:M\to T_xM$ be the exponential and the logarithmic maps at $x$, respectively. The latter may not be defined on the whole tangent space, but it is always defined at least on a neighborhood of the origin. To provide intuition, at a given point $x\in M$, and for a given vector $u\in T_xM$, $\Exp_x(u)$ is the Riemannian analog of ``$x+u$" in the Euclidean space: From point $x$, add a vector $u$, whose direction indicates in which direction to move, and whose norm (given by the Riemannian metric $g$) indicates at which distance, in order to attain the point denoted by $\Exp_x(u)$. On the opposite, given $x,y\in M$, $\Log_x(y)$ (when it is well-defined) is the Riemannian analog of ``$y-x$", that is, from $x$, which vector in $T_xM$ led to move to $y$. 

For all $x\in M$, we denote by $\langle\cdot,\cdot\rangle_x$ the scalar product on $T_xM$ inherited from $g$ and by $\|\cdot\|_x$ the corresponding norm.

Let $x,y\in M$ and $\gamma\in\Gamma_{x,y}$ be a geodesic (in the sense of Section~\ref{sec:geodspaces}) from $x$ to $y$. If $x$ and $y$ are close enough, $\gamma$ is unique \cite[Proposition 6.11]{lee2018introduction} and in that case, for all $t\in [0,1]$, $\gamma(t)=\Exp_x(t\dot\gamma(0))$, where $\dot\gamma(0)\in T_xM$ is the derivative of $\gamma$ at $t=0$, given by $\dot\gamma(0)=\Log_x(y)$. We denote the parallel transport map from $x$ to $y$ along $\gamma$ as $\pi_{x,y}:T_xM\to T_yM$, without specifying its dependence on $\gamma$ (which is non-ambiguous if $y$ is close enough to $x$). The map $\pi_{x,y}$ is a linear isometry, i.e., $\langle u,v\rangle_x=\langle \pi_{x,y}(u),\pi_{x,y}(v)\rangle_y$, for all $u,v\in T_xM$. Moreover, it holds that $\pi_{x,y}(\dot\gamma(0))=-\dot\gamma(1)$. Intuitively, parallel transport is an important tool that allows to compare vectors from different tangent spaces, along a given path on the manifold. 

For further details and properties on Riemannian manifolds used in the mathematical development of this work, we refer the reader to Appendix~\ref{sec:AppendixRiemann}.

\subsection{$M$-estimation}

Let $\mathcal Z$ be some abstract space, equipped with a $\sigma$-algebra $\mathcal F$ and a probability measure $P$. Let $\phi:\mathcal Z\times M\to\R$ satisfy the following: 
\begin{enumerate}[(i)]
	\item For all $x\in M$, the map $\phi(\cdot,x)$ is measurable and integrable with respect to $P$;
	\item For $P$-almost all $z\in\mathcal Z$, $\phi(z,\cdot)$ is convex. 
\end{enumerate}
Finally, let $Z,Z_1,Z_2,\ldots$ be i.i.d. random variables in $\mathcal Z$ with distribution $P$ and set the functions 
\begin{equation*}
	\Phi(x)=\E[\phi(Z,x)] \quad \mbox{ and } \quad \hat\Phi_n(x)=n^{-1}\sum_{i=1}^n \phi(Z_i,x),
\end{equation*}
for all $x\in M$ and for all positive integers $n$. Finally, we denote by $M^*$ the set of minimizers of $\Phi$. Note that by convexity of $\Phi$, the minimizing set $M^*$ is convex, i.e., for all $x,y\in M^*$ and for all $\gamma\in\Gamma_{x,y}$, $\gamma([0,1])\subseteq M^*$. 

%\paragraph{Example: Location estimation}

Important examples include the case where $\mathcal Z=M$ and $\phi$ is a function of the metric, which yields location estimation. Specifically, assume that $\mathcal Z=M$ and that $\phi(z,x)=\ell(d(z,x))$, for all $z,x\in M$, where $\ell:[0,\infty)\to [0,\infty)$ is a non-decreasing, convex function. Then, a natural framework to ensure convexity of the functions $\phi(z,\cdot)$ is that of spaces with curvature upper bounds, a.k.a. CAT spaces. Here, we only give an intuitive definition on CAT spaces. We refer the reader to Appendix~\ref{sec:AppendCAT} for a more detailed account on CAT spaces. 

Fix some real number $\kappa$. We say that $(M,d)$ is $\CAT(\kappa)$ (or that it has global curvature upper bound $\kappa$) if all small enough triangles in $M$ are thinner than what they would be in a space of constant curvature $\kappa$, i.e., a sphere if $\kappa>0$, a Euclidean space if $\kappa=0$ and a hyperbolic space if $\kappa<0$. This setup is very natural to grasp the convexity of the distance function to a point. Fix $x_0\in M$. The function $d(\cdot,x_0)$ is convex if and only if for all $y,z\in M$ and all $\gamma\in \Gamma_{y,z}$, $d(\gamma(t),x_0)\leq (1-t)d(y,x_0)+td(z,x_0)$, which controls the distance from $x_0$ to any point in the opposite edge to $x_0$ in a triangle with vertices $x_0,y,z$. In other words, it indicates how thin a triangle with vertices $x_0,y,z$ must be. 

Hence, if $\mathcal Z=M$ is $\CAT(\kappa)$ and has diameter at most $D_\kappa$ (see Appendix~\ref{sec:AppendCAT} for its definition), \cite[Proposition 3.1]{Ohtaconvexity} guarantees that for all $z\in M$, $d(z,\cdot)$ is convex on $M$ and hence, $\phi(z,\cdot)=\ell(\phi(z,\cdot))$ is also convex as soon as $\ell:[0,\infty)\to[0,\infty)$ is non-decreasing and convex itself. In that setup, important examples of functions $\ell$ include:

\begin{enumerate}[(i)]
	\item $\ell(u)=u^2$: Then, a minimizer of $\Phi$ is a barycenter of $P$ and a minimizer of $\hat\Phi_n$ is an empirical barycenter of $Z_1,\ldots,Z_n$
	\item $\ell(u)=u$: Then, a minimizer of $\Phi$ is a geometric median of $P$
	\item More generally, if $\ell(u)=u^p, p\geq 1$, a minimizer of $\Phi$ is called a $p$-mean of $P$
	\item $\ell(u)=u^2$ if $0\leq u\leq c$, $\ell(u)=c(2u-c)$ if $u>c$, where $c>0$ is a fixed parameter. This function was introduced by Huber \cite{huber1992robust} in order to produce low-bias robust estimators of the mean: It provides an interpolation between barycenter and geometric median estimation.
\end{enumerate}

Important examples of CAT spaces include Euclidean spaces, spheres, simply connected Riemannian manifolds with bounded sectional curvature, metric trees, etc. In Appendix~\ref{sec:AppendCAT}, we provide a more detailed list of examples.

\section{Consistency}

This section contains our first main results on the consistency of geodesically convex $M$-estimators. Here, we work under two scenarios. First, we assume that $\phi(z,\cdot)$ is convex and Lipschitz with some constant that does not depend on $z\in\mathcal Z$. This is somewhat restrictive, even though it covers most of the cases mentioned above for location estimation. In the second scenario, we only assume convexity, but we consider the case where $M$ is a Riemannian manifold. We then discuss possible extensions of our findings.

\subsection{Lipschitz case}

Here, we assume that there exists $L>0$ such that $\phi(z,\cdot)$ is $L$-Lipschitz on $M$, for all $z\in\mathcal Z$. For instance, this is satisfied if $M$ has bounded diameter and $\phi=\ell\circ d$, for some locally Lipschitz function $\ell:[0,\infty)$, such as the functions mentioned above for location estimation.

Recall that we denote by $M^*$ the set of minimizers of $\Phi$. The following theorem need not require that $\Phi$ has a unique minimizer: It asserts that any minimizer of $\Phi_n$ will eventually become arbitrarily close to $M^*$ with probability $1$.

\begin{theorem} \label{thm:consistency}
	Let $M$ be a proper geodesic metric space. Assume that $M^*$ is non-empty and bounded. For $n\geq 1$, let $\hat x_n$ be a minimizer of $\hat\Phi_n$. Then, it holds that $d(\hat x_n,M^*)\xrightarrow[n\to\infty]{} 0$ almost surely. 
\end{theorem}

A simple adaptation of the proof shows that the same conclusion applies if $\hat x_n$ is not exactly a minimizer of $\hat \Phi_n$ but rather satisfies $\hat\Phi_n(\hat x_n)\leq \min_{x\in M}\hat\Phi_n(x) +\varepsilon_n$, where $\varepsilon_n$ is any non-negative error term that goes to zero almost surely, as $n\to\infty$.

The proof of Theorem~\ref{thm:consistency} is based on the following extension of \cite[Theorem 10.8]{rockafellar1970convex}. The proof is a straightforward adaptation of that of \cite[Theorem 10.8]{rockafellar1970convex}.

\begin{lemma} \label{lemma:Rockafellar}
	Let $L>0$ and $(f_n)_{n\geq 1}$ be a sequence of $L$-Lipschitz convex functions on $M$. Assume that there is a function $f:M\to \R$ and a dense subset $M_0$ of $M$ such that $f_n(x)\xrightarrow[n\to\infty]{} f(x)$, for all $x\in M_0$. Then, $f$ is convex, $L$-Lipschitz, and $f_n$ converges to $f$ uniformly on any compact subset of $M$. 
\end{lemma}

This lemma is the reason why, in this section, we require $\phi(z,\cdot)$ to be Lipschitz, with a constant that does not depend on $z\in \mathcal Z$. If $M$ was Euclidean, the Lipschitz assumption of Lemma~\ref{lemma:Rockafellar} would not be necessary because on any compact subset, it would be a consequence of the convexity and pointwise convergence of the functions $f_n, n\geq 1$, see \cite[Theorem 10.6]{rockafellar1970convex}. However, to the best of our knowledge, this may not hold in the more general case that we treat here (and may be subject to a further research question).

As a corollary to Lemma~\ref{lemma:Rockafellar}, we now state the following results, which is essential in our proof of Theorem~\ref{thm:consistency}. The proof is deferred to Appendix~\ref{sec:AppendSeqConvexFns}.

\begin{lemma} \label{lemma:RockafellarRandom}

Let $L>0$ and $(F_n)_{n\geq 1}$ be a sequence of real valued, $L$-Lipschitz, convex random functions on $M$. Let $F$ a (possibly random) function on $M$ and assume that for all $x\in M$, $F_n(x)$ converges almost surely (resp. in probability) to $F(x)$. Then, the convergence holds uniformly on any compact subset $K$ of $M$, i.e., $\sup_{x\in K}|F_n(x)-F(x)|$ converges to zero almost surely (resp. in probability). 

\end{lemma}

\begin{myproof}{Theorem \ref{thm:consistency}}

Let $\Phi^*=\min_{x\in M}\Phi(x)$ be the smallest value of $\Phi$ on $M$ and fix some arbitrary $x^*\in M^*$. Fix $\varepsilon>0$ and let $K_\varepsilon=\{x\in M: d(x,M^*)=\varepsilon\}$. Since $M^*$ is bounded, $K_\varepsilon$ is bounded. Moreover, since the distance is continuous, $K_\varepsilon$ is a closed set. Hence, it is compact, since we have assumed $(M,d)$ to be proper. Since $\Phi$ is Lipschitz, it is continuous so it holds that $\eta:=\min_{x\in K_\varepsilon} \Phi(x)-\Phi^*>0$. Moreover, by the law of large numbers, $\Phi_n(x)\xrightarrow[n\to\infty]{} \Phi(x)$ almost surely, for all $x\in M$. Therefore, by Lemma~\ref{lemma:RockafellarRandom}, $\Phi_n$ converges uniformly to $\Phi$ on $K_\varepsilon$ almost surely. So, with probability $1$, $\inf_{x\in K_\varepsilon} \hat\Phi_n(x) > \Phi^*+2\eta/3$ for all large enough $n$. Moreover, also with probability $1$, $\hat\Phi_n(x^*)<\Phi^*+\eta/3$ for all large enough $n$. 

We have established that with probability $1$, it holds simultaneously, for all large enough $n$, that $\inf_{x\in K_\varepsilon} \hat\Phi_n(x)>\hat\Phi_n(x^*)$. Let us show that this, together with the convexity of $\hat\Phi_n$, implies that any minimizer of $\hat \Phi_n$ must be at a distance at most $\varepsilon$ of $M^*$. This will yield the desired result. 

Assume, for the sake of contradiction, that $\hat\Phi_n$ has a minimizer $\tilde x_n$ that satisfies $d(\tilde x_n,M^*)>\varepsilon$. Consider a geodesic $\gamma\in\Gamma_{x^*,\tilde x_n}$. Then, by continuity of the function $d(\cdot,M^*)$, there must be some $t\in [0,1]$ such that $\gamma(t)\in K_\varepsilon$. The convexity of $\hat\Phi_n$ yields the convexity of $\hat\Phi_n\circ\gamma$, which is minimum at $t=1$. Hence, $\hat\Phi_n\circ\gamma$ must be non-increasing, implying that $\inf_{x\in K_\varepsilon} \hat\Phi_n(x)\leq \hat\Phi_n(\gamma(t))\leq \hat\Phi_n(x^*)$, which yields a contradiction. 

\end{myproof}

\vspace{-0.4cm}

\subsection{Riemannian framework}

Now, we prove that, at least in a Riemannian manifold, the Lipschitz condition on $\phi$ can be dropped for the consistency of $\hat x_n$. In this section, we assume that $(M,g)$ is a complete Riemannian manifold and we have the following theorem. 

\begin{theorem} \label{thm:consistencyRiem}
	Let $(M,g)$ be a Riemannian manifold. Assume that $M^*$ is non-empty and bounded. For $n\geq 1$, let $\hat x_n$ be a minimizer of $\hat\Phi_n$. Then, it holds that $d(\hat x_n,M^*)\xrightarrow[n\to\infty]{} 0$ almost surely. 
\end{theorem}

Again, this theorem applies if $\hat x_n$ satisfies $\hat\Phi_n(\hat x_n)\leq \min_{x\in M}\hat\Phi_n(x) +\varepsilon_n$, where $\varepsilon_n$ is any non-negative error term that goes to zero almost surely, as $n\to\infty$.

Once one has Lemma~\ref{lemma:456} below, which builds upon Lemma~\ref{lemma:123}, the proof of Theorem~\ref{thm:consistencyRiem} is very similar to that of Theorem~\ref{thm:consistency}, hence, we omit it.

\begin{lemma}\cite{greene1973subharmonicity} \label{lemma:123}
	Any convex function on a Riemannian manifold is continuous and locally Lipschitz.
\end{lemma}

By adapting the proof of \cite[Theorem 10.8]{rockafellar1970convex}, this lemma yields the following result.

\begin{lemma} \label{lemma:456}
	Let $(f_n)_{n\geq 1}$ be a sequence of convex functions on $M$. Assume that $f_n(x)\xrightarrow[n\to\infty]{}f(x)$, for all $x$ in a dense subset $M_0$ of $M$, where $f:M\to\R$ is some given convex function. Then, $f_n$ is equi-Lipschitz and converges to $f$ uniformly on any compact subset of $M$.
\end{lemma}

%Continuity in Lemma~\ref{lemma:123} appears in \cite[Theorem 3.6 (i)]{udricste1994convex}. However, the proof seems to be incorrect because it implicitly assumes that $f$ is bounded from above, which should be proven. One can find a proof of the continuity of $f$ in \cite{greene1973subharmonicity}. Then, the Lipschitz property follows from a straightforward adaptation of \cite{rockafellar1970convex}.

\begin{remark}
	Lemma~\ref{lemma:123} is key in the proof of consistency if one does not assume that $\phi(z,\cdot)$ is $L$-Lipschitz, for all $z\in\mathcal Z$, with some $L>0$ that does not depend on $z$. As we pointed out in the previous section, we do not know whether this lemma still holds true in more general proper geodesic spaces and we leave this as an open question. 
\end{remark}

\section{Asymptotic normality}

Now, we turn to asymptotic normality of $\hat x_n$. Riemannian manifolds offer a natural and reasonable framework because, on top of their metric structure, they enjoy smoothness which allows to compute first and second order expansions and Gaussian distributions can be defined on their tangent spaces. Hence, in this section, we assume that $(M,g)$ be a complete Riemannian manifold.

\begin{theorem} \label{thm:asymptoticnormality}

Assume that the distribution $P$ and the function $\phi$ satisfy the following assumptions:
\begin{itemize}
	\item $\Phi$ has a unique minimizer $x^*\in M$	
	\item $\Phi$ is twice differentiable at $x^*$ and $\Hess\Phi(x^*)$ is positive definite
	\item There exists $\eta>0$ such that $\E[\|g(Z,x)\|_{x^*}^2]<\infty$, for all $x\in M$ with $d(x^*,x)\leq \eta$, where $g(Z,x)$ is a measurable subgradient of $\phi(Z,\cdot)$ at the point $x$.
\end{itemize}
Then, $\hat x_n$ is asymptotically normal, that is, 
$$\sqrt n\Log_{x^*}(\hat x_n) \xrightarrow[n\to\infty]{(d)} \mathcal N_{T_{x^*}M}(0,V(x^*))$$
where $V(x^*)=S(x^*)^{-1}B(x^*)S(x^*)^{-1}$, $S(x^*)=\Hess\Phi(x^*)$, $B(x^*)=\E[g(Z,x^*)g(Z,x^*)^\top]$ and we denote by $\mathcal N_{T_{x^*}M}(0,V(x^*))$ the normal distribution on the Euclidean space $T_{x^*}M$ with mean $0$ and covariance operator $V(x^*)$.

\end{theorem}

Here, $\Hess\Phi(x^*)$ is the Hessian of $\Phi$ at the point $x^*$ (see Appendix~\ref{sec:AppendixRiemann}). Recall that $\Log_{x^*}(\hat x_n)$ is the Riemannian analog of ``$\hat x_n-x^*$", so Theorem~\ref{thm:asymptoticnormality} is a natural formulation of asymptotic normality of $\hat x_n$ is the Riemannian context.

A close look to the end of the proof will convince the reader that one can assume, without modifying the conclusion of this theorem, that $\hat x_n$ need not be a minimizer of $\hat\Phi_n$, but should satisfy that $\hat\Phi_n(\hat x_n)\leq \min_{x\in M}\hat\Phi_n(x)+o_P(n^{-1/2})$, where $o_P(n^{-1/2})$ is a possibly random error term that goes to zero in probability at a faster rate than $n^{-1/2}$.

\begin{remark}
The statement of Theorem \ref{thm:asymptoticnormality} implicitly assumes the existence of measurable subgradients of $\phi$, that is, a mapping $g:\mathcal Z\times M\to TM$ such that for $P$-almost all $z\in \mathcal Z$ and for all $x\in M$, $g(z,x)\in T_xM$ is a subgradient of $\phi(z,\cdot)$ at the point $x\in M$, and for all $x\in M$, $g(\cdot,x):\mathcal Z\to T_xM$ is measurable. This is actually guaranteed by Lemma \ref{lemma:MeasurableSelections} on measurable selections, see Appendix \ref{sec:AppendixAdditionalLemmas}. 
Note also that the last assumption is automatically guaranteed if $\phi(z,\cdot)$ is locally $L(z)$-Lipschitz at $x^*$, for $P$-almost all $z\in\mathcal Z$, for some $L\in L^2(P)$. 
\end{remark}

The proof of Theorem~\ref{thm:asymptoticnormality} is strongly inspired by \cite{haberman1989concavity} and \cite{niemiro1992asymptotics} but requires subtle tools to account for the non-Euclidean geometry of the problem. The idea of the proof is to show that 
$$\hat\Phi_n(\Exp_{x^*}(u/\sqrt n)) \approx \hat\Phi_n(x^*)+\frac{1}{\sqrt n}\langle u,n^{-1}\sum_{i=1}^n g(Z_i,x^*) \rangle_{x^*} + n^{-1}\langle u,\Hess\Phi(x^*)u \rangle_{x^*}$$ as $n$ grows large, uniformly in $u\in T_{x^*}M$. Hence, the minimizer of the left-hand side which, by definition, is $\Log_{x^*}\hat x_n$, must be close to the minimizer of the right-hand side, which has a closed form, and which is asymptotically normal, thanks to the central limit theorem. The main difficulty is to handle $\hat\Phi_n(\Exp_{x^*}(u/\sqrt n))$, which is not a convex function of $u\in T_{x^*}M$, even though $\hat\Phi_n$ is convex on $M$. Lemma \ref{lemma:convex_generalized} below is our most technical result, and is the key argument to adapt the techniques of \cite{haberman1989concavity,niemiro1992asymptotics}
to the Riemannian setup.

\begin{lemma} \label{lemma:convex_generalized}

Let $(M,g)$ be a Riemannian manifold and $x_0\in M$. There exists $r>0$ such that the following holds. 
Let $L>0$. There exists $\alpha>0$ such that for all convex functions $f:M\to\R$ that are $L$-Lipschitz on $B(x_0,r)$, the map $$f\circ\Exp_{x_0}+\frac{\alpha}{2}\|\cdot\|_{x_0}^2:T_{x_0}M\to\R$$ is convex on $\{u\in T_{x_0}M:\|u\|_{x_0}\leq r/2\}$, where $\|\cdot\|_{x_0}$ is the norm induced by $g$ on the tangent space $T_{x_0}M$. 
 
\end{lemma}

In the proof of this lemma, we first consider the case where $f$ is smooth and then proceed by approximation of convex functions by smooth functions, leveraging \cite[Theorem 2]{greene1973subharmonicity}. \newline

\begin{myproof}{Theorem \ref{thm:asymptoticnormality}}

Fix $u\in T_{x^*}M$. For all $i=1,\ldots,n$, let 
$$W_{i,n}=\phi(Z_i,\Exp_{x^*}(u/\sqrt n))-\phi(Z_i,x^*)-\frac{1}{\sqrt n}\langle u,g(Z_i,x^*)\rangle_{x^*}.$$ 
The definition of subgradients yields that $W_{i,n}\geq 0$. 
Let $\pi_n$ be the parallel transport from $x^*$ to $x_n:=\Exp_{x^*}(u/\sqrt n)$ through the geodesic $\gamma(t)=\Exp_{x^*}(tu/\sqrt n), t\in [0,1]$. Then, $x^*=\Exp_{x_n}(-\pi_n^{-1}(u)/\sqrt n)$. Thus, one has the following:
\begin{align}
	W_{i,n} & \leq \langle \pi_n(u)/\sqrt n,g(Z_i,x_n)\rangle_{x_n} - \langle u/\sqrt n,g(Z_i,x^*)\rangle_{x^*} \nonumber \\
	& = \langle u/\sqrt n, \pi_n^{-1}\left(g(Z_i,\Exp_{x^*}(u/\sqrt n)\right)-g(Z_i,x^*)\rangle_{x^*},
\end{align}
where we used the fact that $\pi_n$ is an isometry in the last equality.

Therefore, $\E[W_{i,n}^2]\leq n^{-1}\E[Y_n^2]$, where 
$$Y_n=\langle u, \pi_n^{-1}\left(g(Z_1,\Exp_{x^*}(u/\sqrt n))\right)-g(Z_1,x^*)\rangle_x.$$
Note that $\E[Y_n^2]$ is finite if $n$ is large enough, thanks to the last assumption of the theorem. 

The inequalities above yield that $Y_n\geq 0$ for all $n\geq 1$ and by Lemma \ref{lemma:monotonsubgr}, the sequence $(Y_n)_{n\geq 1}$ is non-increasing. Therefore, $Y_n$ converges almost surely to a non-negative random variable $Y$. Let us now prove that $Y=0$ almost surely. In order to do so, let us prove that $\E[Y_n]\xrightarrow[n\to\infty]{} 0$. Since, by the monotone convergence theorem, $\E[Y_n]$ must converge to $\E[Y]$, we will obtain that $\E[Y]=0$. This, combined with the fact that $Y\geq 0$ almost surely, will yield the desired result. 

Fix $x\in M$ with $d(x,x^*)\leq \eta$. Then, for all $v\in T_xM$, $\phi(Z_1,\Exp_x(v))\geq \phi(Z_1,x)+\langle v,g(Z_1,x) \rangle_x$. Taking the expectation readily yields that $\E[g(Z_1,x)]\in \partial \Phi(x)$. Hence, for all integers $n$ that are large enough so $\|u/\sqrt n\|_{x^*}\leq \eta$, it holds that $g_n:=\E[g(Z_1,\Exp_{x^*}(u/\sqrt n)] \in \partial \Phi(\Exp_{x^*}(u/\sqrt n))$. Therefore, 
$$\E[Y_n]=\langle u,\pi_n^{-1}(g_n)-\nabla\Phi(x^*) \rangle_{x^*} \xrightarrow[n\to\infty]{} 0$$
since $\Phi$ is twice differentiable at $x^*$.

Finally, we obtain that $Y=0$ almost surely. Now, $(Y_n)_{n\geq 0}$ is a non-increasing, non-negative sequence, so the monotone convergence theorem yields that $\E[Y_n^2]$ goes to zero as $n\to\infty$, so that
$$\var\left(\sum_{i=1}^n W_{i,n}\right) = \sum_{i=1}^n \var(W_{i,n}) \leq \sum_{i=1}^n \E[W_{i,n}^2] \leq \E[Y_n^2] \xrightarrow[n\to\infty]{} 0.$$

Therefore, by Chebychev's inequality, $\sum_{i=1}^n (W_{i,n}-\E[W_{i,n}])$ converges in probability to zero, as $n\to\infty$, that is, 
\begin{align*}
	& n\left(\hat\Phi_n(\Exp_{x^*}(u/\sqrt n))-\hat\Phi_n(x^*)\right)-\frac{1}{\sqrt n}\langle u,\sum_{i=1}^n g(Z_i,x^*) \rangle_{x^*} \\
	& \quad \quad- n\left(\Phi(\Exp_{x^*}(u/\sqrt n))-\Phi(x^*) - \langle u,\nabla\Phi(x^*)\rangle_{x^*}\right) \xrightarrow[n\to\infty]{}0
\end{align*}
in probability. Since $\Phi$ is twice differentiable at $x^*$, we obtain 
\begin{equation}
	n\left(\hat\Phi_n(\Exp_{x^*}(u/\sqrt n))-\hat\Phi_n(x^*)\right)-\frac{1}{\sqrt n}\langle u,\sum_{i=1}^n g(Z_i,x^*) \rangle_{x^*} - \langle u,\Hess\Phi(x^*)u \rangle_{x^*} \xrightarrow[n\to\infty]{}0 \label{eq:ConvergenceProba}
\end{equation}
in probability. For $n\geq 1$ and $u\in T_{x^*}M$
, denote by 
$$F_n(u)=n\left(\hat\Phi_n(\Exp_{x^*}(u/\sqrt n))-\hat\Phi_n(x^*)\right)-\frac{1}{\sqrt n}\langle u,\sum_{i=1}^n g(Z_i,x^*) \rangle_{x^*}$$ 
and by 
$$F(u)=\langle u,\Hess\Phi(x^*)u \rangle_{x^*}.$$ 
We have shown that $F_n(u)$ converges in probability to $F(u)$ as $n\to\infty$, for any fixed $u\in T_{x^*}M$. Now, we would like to make use of Lemma \ref{lemma:RockafellarRandom}. However, the functions $F_n$ are not guaranteed to be convex. We adjust these functions using lemma \ref{lemma:convex_generalized}.

Let $R>0$ be any fixed positive number. Let $K=B_{x^*}(0,R)$ be the ball in $T_{x^*}M$ centered at the origin, with radius $R$ and let $\tilde K=B(x^*,R)$. Then, $K$ and $\tilde K$ are both compact. By the law of large numbers, $\hat\Phi_n(x)$ converges almost surely to $\Phi(x)$, for all $x\in \tilde K$. Therefore, by Lemma~\ref{lemma:RockafellarRandom}, $\sup_{x\in\tilde K}|\hat\Phi_n(x)-\Phi(x)|\xrightarrow[n\to\infty]{} 0$ almost surely. Hence, Lemma \ref{lemma:456} yields the existence of a (possibly random) $L>0$ such that the following holds with probability $1$: 
$$|\hat\Phi_n(x)-\hat\Phi_n(y)|\leq Ld(x,y), \forall n\geq 1, \forall x,y\in \tilde K.$$

Now, let $r>0$ be as in Lemma~\ref{lemma:convex_generalized}. Then, there exists a (possibly random) $\alpha>0$ such that with probability $1$, for all $n\geq 1$, $\hat\Phi_n\circ\Exp_{x^*}+\frac{\alpha}{2}\|\cdot\|_{x^*}^2$ is convex on $K$. 

Denote by $\tilde F_n(u)=F_n(u)+\frac{\alpha}{2}\|u\|_{x^*}^2$ and $\tilde F(u)=F(u)+\frac{\alpha}{2}\|u\|_{x^*}^2$, for all $u\in K$ and $n\geq 1$. For all large enough $n\geq 1$ (such that $R/\sqrt n\leq r$), each $F_n$ is convex on $K$ with probability $1$, and we trivially have that $\tilde F_n(u)\xrightarrow[n\to\infty]{}\tilde F(u)$ in probability, for all $u\in K$. Therefore, Lemma~\ref{lemma:RockafellarRandom} yields that $\sup_{u\in K} |F_n(u)-F(u)|=\sup_{u\in K} |\tilde F_n(u)-\tilde F(u)| \xrightarrow[n\to\infty]{} 0$ in probability. 

Finally, we have shown that for any $R>0$, 
\begin{equation}
	\sup_{\|u\|_{x^*}\leq R} \left|n\left(\hat\Phi_n(\Exp_{x^*}(u/\sqrt n))-\hat\Phi_n(x^*)\right)-\frac{1}{\sqrt n}\langle u,\sum_{i=1}^n g(Z_i,x^*) \rangle_{x^*} - \langle u,\Hess\Phi(x^*)u \rangle_{x^*}\right| \xrightarrow[n\to\infty]{}0 \label{eq:lastconv}
\end{equation}
in probability. 

By definition of $\hat x_n$, $\hat\Phi_n(\Exp_{x^*}(u/\sqrt n))$ is minimized for $u=\sqrt n \Log_{x^*}(\hat x_n)$. Moreover, one easily checks that $\frac{1}{\sqrt n}\langle u,\sum_{i=1}^n g(Z_i,x^*) \rangle_{x^*} - \langle u,\Hess\Phi(x^*)u \rangle_{x^*}$ is minimized for 
$$u=-\Hess\Phi(x^*)^{-1}\frac{1}{\sqrt n}\sum_{i=1}^n g(Z_i,x^*).$$

Note that $g(Z_1,x^*), \ldots,g(Z_n,x^*)$ are i.i.d. random vectors in $T_{x^*}M$. They are centered, since $\E[g(Z_1,x^*)]=\nabla\Phi(x^*)=0$ by the first order condition, and they have a second moment, thanks to the last assumption of the theorem. Therefore, the multivariate central limit theorem  yields that 
\begin{equation} \label{eq:CLT1}
\frac{1}{\sqrt n}\sum_{i=1}^n g(Z_i,x^*)\xrightarrow[n\to\infty]{} \mathcal N_{T_{x^*}M}\left(0,B(x^*)\right)
\end{equation}
in distribution. In particular, the sequence $\frac{1}{\sqrt n}\sum_{i=1}^n g(Z_i,x^*)$ is bounded in probability, so one can choose $R$ in  \eqref{eq:lastconv} so that $\frac{1}{\sqrt n}\sum_{i=1}^n g(Z_i,x^*) \in B_{x_0}(0,R)$ for all $n\geq 1$, with arbitrarily large probability. 

From \eqref{eq:lastconv}, it is easy to obtain that
$\sqrt n \Log_{x^*}(\hat x_n)-\Hess\Phi(x^*)^{-1}\frac{1}{\sqrt n}\sum_{i=1}^n g(Z_i,x^*)$ converges in probability to zero, as $n\to\infty$. Therefore, by Slutsky's theorem, $\sqrt n \Log_{x^*}(\hat x_n)$ has the same limit in distribution as $\Hess\Phi(x^*)^{-1}\frac{1}{\sqrt n}\sum_{i=1}^n g(Z_i,x^*)$, which concludes the proof thanks to \eqref{eq:CLT1}.

\end{myproof}

\begin{myproof}{Lemma~\ref{lemma:convex_generalized}}

Fix $r>0$ (to be chosen) and denote by $K=B(x_0,r)$, which is compact by Hopf-Rinow theorem \cite[Chapter 7, Theorem 2.8]{doCarmo1992riemannian}. Let $f:M\to\R$ be convex and $L$-Lipschitz ($L>0$) on $K$ and set $F=f\circ\Exp_{x_0}:T_{x_0}M\to\R$.  

We split the proof of this lemma into steps.

\vspace{2mm}
\paragraph{Step 1: Case where $f$ is smooth.}

First, assume that $f$ is smooth. Then, $F$ is a smooth function. In the sequel, we denote by $\Hess f$ the Hessian of $f$ and by $\bar \Hess F$ the Hessian of $F$ (i.e., we use $\Hess$ for Hessian of functions defined on $M$ and $\bar\Hess$ for Hessian of functions defined on the tangent space $T_{x_0}M$).
Fix $u,v\in T_{x_0}M$ and let $\phi(t)=F(u+tv)$, for $s\in\R$. Denote by $\eta(s,t)=\Exp_{x_0}(t(u+sv))$, for all $s,t\in\R$. The Hessian of $F$ is given by $\bar \Hess F(u)(v,v)=\phi''(0)$. For all $s\in \R$, 
$$\phi'(s)=\langle \nabla f(\eta(s,1)),\frac{\partial \eta}{\partial s}(s,1)\rangle_{\eta(s,1)}$$
and 
\begin{equation} \label{eq:Hessian}
	\phi''(s)=\langle \Hess f(\eta(s,t))\frac{\partial \eta}{\partial s}(s,1),\frac{\partial \eta}{\partial s}(s,1)  \rangle_{\eta(s,1)} + \langle \nabla f(\eta(s,1)),D_s\frac{\partial \eta}{\partial s}(s,1)\rangle_{\eta(s,1)}
\end{equation}
where $D_s$ is the covariant derivative along the path $\eta(\cdot,1)$.
Note that the first term on the right-hand side of \eqref{eq:Hessian} is non-negative, due to the convexity of $f$ (see \cite[Theorem 6.2]{udricste1994convex}). Therefore, 
\begin{align}
	\phi''(0) & \geq \langle \nabla f(\eta(0,1)),D_s\frac{\partial \eta}{\partial s}(0,1)\rangle_{\eta(0,1)} \nonumber \geq - \|\nabla f(\eta(0,1))\| \|D_s\frac{\partial \eta}{\partial s}(0,1)\|\nonumber \\
	& \geq -L\|D_s\frac{\partial \eta}{\partial s}(0,1)\| \label{eq:Hessian1} 
\end{align}
where the second inequality follows from Cauchy-Schwarz inequality and the last one follows from the assumption that $f$ is $L$-Lipschitz on $K$. 

For all $w,w'\in T_{x_0}M$, let $\Gamma_{w,w'}(s,t)=\Exp_{x_0}(t(w+sw')), s,t\in [0,1]$. In particular, $\eta=\Gamma_{u,v}$. Then, $D_s\frac{\partial \Gamma_{u,v}}{\partial s}(0,1)=\|v\|^2 D_s\frac{\partial \Gamma_{u,\tilde v}}{\partial s}(0,1)$, where $\tilde v=\frac{v}{\|v\|_{x_0}}$. Since the map $w,w'\mapsto D_s\frac{\partial \Gamma_{w,w'}}{\partial s}(0,1)$ is continuous, there exists a constant $C>0$ (that only depends on $r$ and on the geometry of $(M,g)$) such that 
$$\|D_s\frac{\partial \Gamma_{w,w'}}{\partial s}(0,1)\|\leq C,$$
for all $u,v\in T_{x_0}M$ with $\|u\|_{x_0}\leq r$ and $\|v\|_{x_0}\leq 1$. Hence, \eqref{eq:Hessian1} yields that $\DS \phi''(0)\geq -LC\|v\|_{x_0}^2$. Therefore, by setting $\alpha_1=LC$, we have established that $F+\frac{\alpha_1}{2}\|\cdot\|_{x_0}^2$ is convex on $\{u\in T_{x_0}M:\|u\|_{x_0}\leq r\}$.

\vspace{2mm}
\paragraph{Step 2: Case where $f$ is no longer assumed to be smooth.}

\subparagraph{Step 2.1: Smooth approximation of $f$.} 
For all positive integers $n$, let $f_n:M\to\R$ be the function defined by 
$$f_n(x)=n^d \int_{T_xM}f(\Exp_x(v))k(n\|v\|_x)\diff\Omega_x(v)$$
where $k:\R\to\R$ is a non-negative, smooth function supported on $[0,1]$ with $\int_\R k(t)\diff t=1$ and $\Omega_x$ is the pushforward measure of the Riemannian volume onto $T_xM$ by the exponential map $\Exp_x$. By \cite[Theorem 2]{greene1973subharmonicity}, each $f_n$ is smooth and the sequence $(f_n)_{n\geq 1}$ converges to $f$ uniformly on $K$ and satisfies
$$\liminf_{n\to\infty} \inf_{x\in K}\lambda_{\min}(\Hess f_n(x)) \geq 0,$$
where $\lambda_{\min}$ stands for the smallest eigenvalue.  Fix $\varepsilon>0$, that we will then choose to be small enough. Without loss of generality, assume that $\inf_{x\in K}\lambda_{\min}(\Hess f_n(x)) \geq -\varepsilon$, for all $n\geq 1$ (since one might as well forget about the first terms of the sequence $(f_n)_{n\geq 1}$). 

\vspace{2mm}
\subparagraph{Step 2.2: $f_n$ is Lipschitz on $B(x_0,r/2)$.}

Let us show that$f_n$ is Lipschitz on $B(x_0,r/2)$, so long as $n$ is large enough. Let $x,y\in B(x_0,r/2)$ and let $\pi_{x,y}:T_xM\to T_yM$ be the parallel transport map from $x$ to $y$. Since $\pi_{x,y}$ is an isometry, it maps the measure $\Omega_x$ to $\Omega_y$, hence, it holds that 
$$f_n(y)=n^d \int_{T_xM}f(\Exp_y(\pi_{x,y}(v)))k(n\|v\|_x)\diff\Omega_x(v).$$
Thus,
\begin{align}
	|f_n(x)-f_n(y)| & = n^d\left|\int_{T_xM} (f(\Exp_x(v))-f(\Exp_y(\pi_{x,y}(v)))) k(n\|v\|_{x})\right|\diff\Omega_x(v) \nonumber \\
	& \leq n^d \int_{T_xM} \left|f(\Exp_x(v))-f(\Exp_y(\pi_{x,y}(v)))\right| k(n\|v\|_{x}) \diff\Omega_x(v) \nonumber \\
	& \leq Ln^d \int_{T_xM} d(\Exp_x(v),\Exp_y(\pi_{x,y}(v)) k(n\|v\|_{x}) \diff\Omega_x(v) \label{eq:fnLip}
\end{align}
where the last inequality holds as soon as $n$ is large enough (i.e., $n\geq 2/r$) since $f$ is $L$-Lipschitz on $K=B(x_0,r)$ and in the last integral, only the vectors $v\in T_xM$ with $\|v\|_{x_0}\leq n^{-1}\leq r/2$ contribute to the integral. 
Finally, the map $(x,y,v)\mapsto \Exp_y(\pi_{x,y}(v))$ is smooth, hence it is Lipschitz on any compact set, so there exists $C>0$ that is independent of $x$ and $y$ such that for all $v\in T_xM$ with $\|v\|_{x_0}\leq r/2$, $d(\Exp_x(v),\Exp_y(\pi_{x,y}(v))\leq Cd(x,y)$. Therefore, \eqref{eq:fnLip} yields 
\begin{equation}
	|f_n(x)-f_n(y)| \leq LCd(x,y)n^d \int_{T_xM} k(n\|v\|_{x}) \diff\Omega_x(v) = LCd(x,y), 
\end{equation}
for all $x,y\in B(x_0,r/2)$. Thus, $f_n$ is $CL$-Lipschitz on $B(x_0,r/2)$, for all integers $n\geq 2/r$. 

\vspace{2mm}
\subparagraph{Step 2.3: Add a fixed function to each $f_n$ to make them convex.}

Assume that $r$ and $\varepsilon$ are chosen according to Lemma \ref{lemma:ExistStrConv} below and let $\phi$ be the function given in that lemma. Then, for all $n\geq 1$, $f_n+\phi$ is convex on $K$, since its Hessian is positive semi-definite at any $x\in K$. 

\vspace{2mm}
\subparagraph{Step 2.4: Apply the smooth case (Step 1) to each $f_n+\phi$ on $B(x_0,r/2)$.}

For all $n\geq 1$, $f_n+\phi$ is smooth. Moreover, $f_n$ is $CL$-Lipschitz on $B(x_0,r/2)$, where $C>0$ does not depend on $f$, and $\phi$ is smooth so it is $L'$-Lipschitz on $B(x_0,r/2)$, for some $L'>0$ that does not depend on $f$ either. Therefore, applying the result proven in Step 1, we obtain that there exists $\alpha_2>0$ that does not depend on $f$, such that for all large enough integers $n$, $f_n\circ\Exp_{x_0}+\phi\circ\Exp_{x_0}+\frac{\alpha_2}{2}\|\cdot\|_{x_0}^2$ is convex on $\{u\in T_{x_0}M: \|u\|_{x_0}\leq r/2\}$. By taking the limit as $n\to\infty$, we obtain that $f\circ\Exp_{x_0}+\phi\circ\Exp_{x_0}+\frac{\alpha_2}{2}\|\cdot\|_{x_0}^2$ is convex on $\{u\in T_{x_0}M: \|u\|_{x_0}\leq r/2\}$. 

Since $\phi$ is smooth, so is $\phi\circ\Exp_{x_0}$, so there exists $\beta>0$ such that $\lambda_{\max}\left(\bar\Hess(\phi\circ\Exp_{x_0})(u)\right)\leq \beta$, for all $u\in T_{x_0}M$ with $\|u\|_{x_0}\leq r/2$. Therefore, by setting $\alpha_3=\alpha_2+\beta$, we obtain that $f\circ\Exp_{x_0}+\frac{\alpha_3}{2}\|\cdot\|_{x_0}^2$ is convex on $\{u\in T_{x_0}M: \|u\|_{x_0}\leq r/2\}$. 

\vspace{2mm}
\subparagraph{Step 3: Conclusion.}

We can now combine all cases by taking $\alpha=\max(\alpha_1,\alpha_3)$.

\end{myproof}

\begin{lemma} \label{lemma:ExistStrConv}

	Let $(M,g)$ be a complete Riemannian manifold without boundary and $x_0\in M$. There exists $r,\varepsilon>0$ and a smooth function $\phi:M\to\R$ such that $\lambda_{\min} (\Hess\phi(x)) \geq \varepsilon$ for all $x\in B(x_0,r)$. 
	
\end{lemma}

\begin{proof}

	Let $r_1$ such that $B(x_0,r_1)$ is simply connected. For all $x\in M$ and for all linearly independent vectors $u,v\in T_xM$ with $\|u\|_x=\|v\|_x=1$, let $\kappa(x,u,v)$ be the sectional curvature of $M$ at $x$, along the $2$-plane spanned by $u$ and $v$ (see \cite[Proposition 3.1 and Definition 3.2]{doCarmo1992riemannian}). Then, $\kappa$ is smooth, hence, there exists $\kappa_0<\infty$ that is an upper bound on all sectional curvatures of $M$ at points $x\in B(x_0,r_1)$, which is compact. Therefore, Lemma~\ref{lemma:RiemCAT} yields that $B(x_0,r_1)$ is $\CAT(\kappa_0)$. The desired result follows by taking: $r=\min(r_1,D_{\kappa_0}/4)$ and $\phi=\frac{\varepsilon}{2\alpha_{2r,\kappa_0}}d(\cdot,x_0)^2$, where $D_{\kappa_0}$ and $\alpha_{2r,\kappa_0}$ are defined in Lemma~\ref{lemma:convCAT}, see Appendix \ref{sec:AppendCurvBounds}.

\end{proof}

\bibliographystyle{plain}
\bibliography{Biblio}

\newpage 

\appendix

\section{On measurable selections of subgradients} \label{sec:AppendixAdditionalLemmas}

\begin{lemma} \label{lemma:MeasurableSelectionsSubgradients}
	Let $(\mathcal Z,\mathcal F,P)$ be a probability space and $(M,g)$ be a Riemannian manifold without boundary. Let $\phi:\mathcal Z\times M\to\R$ be a function such that 
	\begin{itemize}
		\item $\phi(z,\cdot)$ is convex for $P$-almost all $z\in \mathcal Z$;
		\item $\phi(\cdot,x)$ is measurable for all $x\in M$.
	\end{itemize}
Then, there exists a map $g:\mathcal Z\times M\to TM$ such that
	\begin{itemize}
		\item $g(z,x)\in T_xM$ is a subgradient of $\phi(z,\cdot)$ at $x$, for all $x\in M$, for $P$-almost all $z\in\mathcal Z$;
		\item $g(\cdot,x)$ is measurable, for all $x\in M$.
	\end{itemize}
\end{lemma}

\begin{proof}

Let $N\in \mathcal F$ be such that $P(N)=0$ and $\phi(z,\cdot)$ is convex on $M$ for all $z\in \mathcal Z\setminus N$. Set $\tilde{\mathcal Z}=\mathcal Z\setminus N$. It is enough to prove that for all $x\in M$, there exists a measurable function $g(\cdot,x)$ on $M$ such that $g(z,x)$ is a subgradient of $\phi(z,\cdot)$ at $x$, for all $z\in \tilde Z$. 

Fix $x_0\in M$. For all $z\in\mathcal Z$, let $\Gamma(z)$ be the set of all subgradients of $\phi(z,\cdot)$ at the point $x_0$. By  \cite[Theorem 4.5]{udricste1994convex}, $\Gamma(z)\neq\emptyset$, for all $z\in \tilde{\mathcal Z}$. Therefore, it is enough to show that the restriction of $\Gamma$ to $\tilde{\mathcal Z}$ is weakly measurable, i.e., that $\{z\in \tilde{\mathcal Z}:\Gamma(z)\cap U\neq\emptyset\}\in\mathcal F$, for all open subsets $U\subseteq T_{x_0}M$ (see \cite{himmelberg1982measurable} for the definition).

One can wrete $\Gamma(z)$ as 
\begin{align*}
	\Gamma(z) & = \bigcap_{v\in T_{x_0}M} \{u\in T_{x_0}M: \phi(z,\Exp_{x_0}(v))\geq \phi(z,x_0)+\langle u,v\rangle_{x_0}\} \\
	& = \bigcap_{v\in E_0} \{u\in T_{x_0}M: \phi(z,\Exp_{x_0}(v))\geq \phi(z,x_0)+\langle u,v\rangle_{x_0}\},
\end{align*}
where $E_0$ is some fixed countable, dense subset of $T_{x_0}M$. The last equality follows from the continuity of $\phi(z,\cdot)$, by Lemma~\ref{lemma:123}. For all $v\in E_0$, let $\Gamma_v(z)=\{u\in T_{x_0}M: \phi(z,\Exp_{x_0}(v))\geq \phi(z,x_0)+\langle u,v\rangle_{x_0}\}$. 

Since each of the countably many $\Gamma_v(z), v\in E_0$, is closed, \cite[Corollary 4.2]{himmelberg1982measurable} shows that it is enough to show that each $\Gamma_v$ is weakly measurable. Fix $v\in E_0$ and let $U$ be an open subset of $T_{x_0}M$. Then, for all $z\in \tilde{\mathcal Z}$, 
\begin{align*}
	\Gamma_v(z)\cap U\neq\emptyset & \iff \exists u\in U, \phi(z,\Exp_{x_0}(v))\geq \phi(z,x_0)+\langle u,v\rangle_{x_0} \\
	& \iff \exists u\in U\cap E_0, \phi(z,\Exp_{x_0}(v))\geq \phi(z,x_0)+\langle u,v\rangle_{x_0}.
\end{align*}
The last equivalence follows from the fact that since $U$ is open, $U\cap E_0$ is dense in $U$, and if some $u\in U$ satisfies $\phi(z,\Exp_{x_0}(v))\geq \phi(z,x_0)+\langle u,v\rangle_{x_0}$, then, any $u'\in U$ of the form $u'=u-v'$, for any $v'\in T_{x_0}M$ with small enough norm such that $\langle v,v'\rangle_{x_0}\leq 0$. There is at least one such $u'$ that falls in $U\cap E_0$. 

Finally, one obtains:
\begin{equation*}
	\{z\in\tilde{\mathcal Z}: \Gamma_v(z)\cap U\neq\emptyset\} = \bigcup_{u\in E_0} \{z\in\tilde{\mathcal Z}:\phi(z,\Exp_{x_0}(v))\geq \phi(z,x_0)+\langle u,v\rangle_{x_0}\},
\end{equation*}
which is in $\mathcal F$ since $\phi(\cdot,x)$ is assumed to be measurable for all $x\in M$. 

One concludes using Lemma~\ref{lemma:MeasurableSelections} below. 

\end{proof}

\begin{lemma} \label{lemma:MeasurableSelections}

	Let $(\mathcal Z,\mathcal F,P)$ be a probability space and $E$ a Polish space. Let $\Gamma$ be a function that maps any $z\in \mathcal Z$ to a closed subset of $E$. Assume the existence of $N\in \mathcal F$ with $P(N)=0$ and such that $\Gamma(z)\neq\emptyset$ for all $z\in\mathcal Z\setminus N$. Assume further that for all open subsets $U$ of $E$, $\{z\in\mathcal Z\setminus N:\Gamma(z)\cap U\neq \emptyset\} \in \mathcal F$. Then, there exists a function $g:\mathcal Z\to E$ such that $g(z)\in\Gamma(z)$, for $P$-almost all $z\in\mathcal Z$. 

\end{lemma}

\begin{proof}

Consider the probability space $(\tilde{\mathcal Z},\tilde{\mathcal F},\tilde P)$ defined by setting 
\begin{itemize}
	\item $\tilde{\mathcal Z}=\mathcal Z\setminus N$;
	\item $\tilde{\mathcal F}=\{F\setminus N:F\in\mathcal F\}$;
	\item $\tilde P$ as the restriction of $P$ to $(\tilde{\mathcal Z},\tilde{\mathcal F})$. 
\end{itemize}
Then, for all $z\in \mathcal Z$, $\Gamma(z)\neq\emptyset$ so, thanks to \cite[Theorem 4.1]{wagner1977survey}, one can find a measurable selection $g$ of the restriction of $\Gamma$ to $\tilde{\mathcal Z}$. By setting $g(z)$ to be any constant value in $E$ for $z\in N$, we obtain a function $g:\mathcal Z\to E$ that satisfies the requirement. 

\end{proof}

\section{Elements on Riemannian manifolds and on geodesic convexity} \label{sec:AppendixRiemann}

In this section, let $(M,g)$ be a complete $d$-dimensional Riemannian manifold.. That is, $M$ is a $d$-dimensional smooth manifold and the Riemannian metric $g$ equips each tangent space $T_xM$, $x\in M$, with a dot product $g_x(u,v)$, $u,v\in T_xM$. Moreover, by Hopf-Rinow theorem \cite[Chapter 7, Theorem 2.8]{doCarmo1992riemannian}, completeness of $(M,d)$ ensures that any two points $x,y\in M$ are connected by at least minimizing geodesic, i.e., a smooth path $\gamma:[0,1]\to M$ such that $\gamma(0)=x, \gamma(1)=y$ and $d(\gamma(s),\gamma(t))=|s-t|d(x,y)$, for all $s,t\in [0,1]$. %For simplicity, when there is no ambiguity, we do not mention the dependence on $x$ and simply write $\langle u,v\rangle$ instead of $g_x(u,v)$.

\subsection{On smooth functions and their derivatives}

If $f:M\to\R$ is a smooth function, its gradient is the map $\nabla f$ such that for all $x\in M$, $\nabla f(x)\in T_xM$ and for all $u\in T_xM$, 
$$\langle\nabla f(x),u\rangle_x = (f\circ\gamma)'(0)$$
where $\gamma:[0,1]\to M$ is any smooth path such that $\gamma'(0)=u$. 
The Hessian of $f$ at any $x\in M$ is represented by the linear map $\Hess f(x):T_xM\to T_xM$ such that for all $u\in T_xM$, 
$$\langle u,\Hess f(x)u\rangle_x=(f\circ\gamma)''(0)$$
where $\gamma$ is a geodesic path with $\gamma(0)=x$, $\gamma'(0)=u$. 

With our notation, any smooth function $f:M\to\R$ has a second order Taylor expansion at any given point $x\in M$ of the following form:

$$f(\Exp_{x}(tu))=f(x)+t\langle\nabla f(x),u\rangle_{x} +\frac{t^2}{2}\langle u,\Hess f(x)u \rangle_{x} +o(t^2)$$
as $t\to 0$, for any fixed $u\in T_{x}M$. 

\subsection{On convex functions}

Recall that a convex function on $M$ is any function $f:M\to \R$ that satisfies $f(\gamma(t))\leq (1-t) f(x)+tf(y)$ for all $t\in [0,1]$, $x,y\in M$ and for all constant speed, minimizing geodesic $\gamma:[0,1]\to M$ from $x$ to $y$. Just as in Euclidean spaces, convex functions are automatically continuous \cite[Theorem 3.6]{udricste1994convex} (note that we only consider convex functions on the whole space $M$, which is assumed to have no boundary). The notion of subgradients can also be extended to the Riemannian case.

\begin{definition}
	Let $f:M\to\R$ be a convex function and $x\in M$. A vector $v\in T_{x}M$ is called a subgradient of $M$ at $x$ if and only if 
	$$f(\Exp_{x}(u))\geq f(x)+\langle v,u\rangle_x,$$
for all $u\in T_{x}(M)$. 
\end{definition}

The set of all subgradients of $f$ at $x\in M$ is denoted by $\partial f(x)$ and \cite[Theorem 4.5]{udricste1994convex} ensures that $\partial f(x)\neq \emptyset$, for all $x\in M$ (again, note that we only consider the case when $M$ has no boundary). For simplicity here, unlike in \cite{udricste1994convex}, subgradients are elements of the tangent space, not its dual.

Subgradients of convex functions satisfy the following monotonicity property.

\begin{lemma} \label{lemma:monotonsubgr}

	Let $f:M\to\R$ be a convex function. Let $x\in M$, $u\in T_xM$ and set $y=\Exp_x(u)$. Let $g_1\in \partial f(x)$ and $g_2\in\partial f(y)$. Then, 
	$$\langle u, \pi_{x,y}^{-1}(g_2)-g_1\rangle_x\geq 0,$$
where $\pi_{x,y}$ is the parallel transport from $x$ to $y$ through the geodesic given by $\gamma(t)=\Exp_x(tu), t\in [0,1]$. 
\end{lemma}

\begin{proof}

First, note that by letting $v=\pi_{x,y}(u)$, it holds that $x=\Exp_{y}(-v)$. By definition of subgradients, one has the following inequalities:
$$f(y)\geq f(x)+\langle u,g_1\rangle_x$$
and
$$f(x)\geq f(y)+\langle -v,g_2\rangle_y = f(y)-\langle u,\pi_{x,y}^{-1}(g_2)\rangle_x,$$
using the fact that $\pi_{x,y}$ is an isometry. Summing both inequalities yields the result.

\end{proof}

%
%\begin{proposition}
%	Let $f:M\to \R$ be a convex function and let $K$ be a compact subset of $M$. Assume that $f$ is $L$-Lipschitz on $K$, for some $L>0$. Then, for all $x\in K$, there exists $u\in \partial f(x)$ with $\|u\|\leq L$. 
%\end{proposition}
%
%\begin{proof}
%
%\end{proof}

\section{On sequences of convex functions in metric spaces} \label{sec:AppendSeqConvexFns}

\subsection{Proof of Lemma \ref{lemma:RockafellarRandom}}

\paragraph{Almost sure convergence}
Since $(M,d)$ is a proper metric space, it is separable \cite[Proposition 2.2.2]{alexander2019alexandrov}. Let $M_0$ be a countable dense subset of $M$. By assumption, $F_n(x)\xrightarrow[n\to\infty]{} F(x)$ almost surely, for all $x\in M_0$. Therefore, with probability one, the convergence holds for all $x\in M_0$, since $M_0$ is countable. Hence, because $M_0$ is also dense, Theorem \ref{lemma:Rockafellar} applies and we obtain the desired result.

\paragraph{Convergence in probability}
Here, we use the following fact.
\begin{lemma} \cite[Lemma 2 (ii) p.67]{chow2003probability}
	Let $X,X_1,X_2,\ldots$ be real valued random variables. Then, $X_n$ converges in probability to $X$ if and only if every subsequence of $(X_n)_{n\geq 1}$ has itself a subsequence that converges almost surely to $X$.
\end{lemma}

Let $K$ be a compact subset of $M$ and consider a subsequence of $(\sup_{x\in K}|F_n(x)-F(x)|)_{n\geq 1}$. Without loss of generality (one could renumber that subsequence), we actually consider the whole sequence $(\sup_{x\in K}|F_n(x)-F(x)|)_{n\geq 1}$ and we show that it admits a subsequence that converges almost surely to zero. Let $M_0$ be a countable dense subset of $M$. Denote the elements of $M_0$ as $x_1,x_2,\ldots$. 

By assumption, $F_n(x_1)$ converges to $F(x_1)$ in probability, therefore, it has a subsequence, say $F_{\phi_1(n)}(x_1)$, that converges to $F(x_1)$ almost surely.

Now, again by assumption, $F_{\phi_1(n)}(x_2)$ converges to $F(x_2)$ in probability, therefore, it has a subsequence, say $F_{\phi_1\circ\phi_2(n)}(x_2)$, that converges to $F(x_2)$ almost surely.

By iterating this procedure, we build, for all integers $k\geq 1$, an increasing function $\phi_k:\N^*\to\N^*$ such that $F_{\phi_1\circ\ldots\circ\phi_k(n)}(x_k)$ converges almost surely to $F(x_k)$, as $n\to\infty$. 

Finally, let $\psi:\N^*\to\N^*$ the function defined by $\psi(n)=\phi_1\circ\ldots\circ\phi_n(n)$, for all $n\geq 1$. Then, by construction, $F_{\psi(n)}(x_k)$ converges almost surely to $F(x_k)$, for all $k\geq 1$. Hence, with probability $1$, $F_{\psi(n)}(x)$ converges almost surely to $F(x)$ for all $x\in M_0$. Lemma \ref{lemma:Rockafellar} yields that $\sup_{x\in K}|F_{\psi(n)}(x)-F(x)|$ converges to zero almost surely, which concludes the proof.

\section{A brief introduction to Alexandrov's curvature} \label{sec:AppendCAT}

In this section, we give the precise definition of CAT spaces. 

\subsection{Model spaces}

First, we introduce a family of model spaces that will allow us to define local and global curvature bounds in the sequel. Let $\kappa\in\R$. 

\paragraph{$\kappa=0$: Euclidean plane} 
Set $M_0=\R^2$, equipped with its Euclidean metric. This model space corresponds to zero curvature, is a geodesic space where geodesics are unique and given by line segments. 

\paragraph{$\kappa>0$: Sphere}
Set $M_\kappa=\frac{1}{\sqrt\kappa}\Sp^2$: This is the $2$-dimensional Euclidean sphere, embedded in $\R^3$, with center $0$ and radius $1/\sqrt\kappa$, equipped with the arc length metric: $d_\kappa(x,y)=\frac{1}{\sqrt\kappa}\arccos(\kappa x^\top y)$, for all $x,y\in M_\kappa$. This is a geodesic space where the geodesics are unique except for opposite points, and given by arcs of great circles. Here, a great circle is the intersection of the sphere with any plane going through the origin. 

\paragraph{$\kappa<0$: Hyperbolic space}
Set $M_\kappa=\frac{1}{\sqrt\kappa}\mathbb H^2$, where $\mathbb H_2=\{(x_1,x_2,x_3)\in\R^3:x_3>0, x_1^2+x_2^2-x_3^2=-1\}$. The metric is given by $d_\kappa=\frac{1}{\sqrt{-\kappa}}\textrm{arccosh}(-\kappa\langle x,y\rangle)$, for all $x,y\in M_\kappa$, where $\langle x,y\rangle=x_1y_1+x_2y_2-x_3y_3$. This is a geodesic space where geodesics are always unique and are given by arcs of the intersections of $M_\kappa$ with planes going through the origin. 

\vspace{4mm} 
For $\kappa\in\R$, let $D_\kappa$ be the diameter of the model space $M_\kappa$, i.e., 
$\DS D_\kappa=\begin{cases} \infty \mbox{ if } \kappa\leq 0 \\ \frac{\pi}{\sqrt\kappa} \mbox{ if } \kappa>0\end{cases}$.

% \subsection{Laws of cosines} \label{sec:LawsCosines}

% Let $\kappa\in\R$ and let $x,y,z\in M_\kappa$ such that $d_\kappa(x,y)+d_\kappa(y,z)+d_\kappa(x,z)<2D_\kappa$. In particular, there is a unique geodesic between any two of the points $x,y,z$. Denote by $a=d_\kappa(y,z)$, $b=d_\kappa(x,z)$, $c=d_\kappa(x,y)$ and let $\theta$ be the angle made by the geodesics segments from $x$ to $y$ and from $x$ to $z$. Then, we have the following laws of cosines, depending on the sign of $\kappa$:
% \begin{enumerate}
% 	\item If $\kappa=0$, $a^2=b^2+c^2-2bc\cos\theta$.
% 	\item If $\kappa>0$, $\cos(\sqrt\kappa a)=\cos(\sqrt\kappa b)\cos(\sqrt\kappa c)+\sin(\sqrt\kappa b)\sin(\sqrt\kappa c)\cos\theta$.
% 	\item If $\kappa<0$, $\cosh(\sqrt{-\kappa} a)=\cosh(\sqrt{-\kappa} b)\cosh(\sqrt{-\kappa} c)-\sinh(\sqrt{-\kappa} b)\sinh(\sqrt{-\kappa} c)\cos\theta$.
% \end{enumerate}

\subsection{Curvature bounds} \label{sec:AppendCurvBounds}

Let $(M,d)$ be a geodesic space. The notion of curvature (lower or upper) bounds for $(M,d)$ is defined by comparing the triangles in $M$ with triangles with the same side lengths in model spaces. 

\begin{definition}
	A (geodesic) triangle in $M$ is a set of three points in $M$ (the vertices) together with three geodesic segments connecting them (the sides). 
\end{definition}

Given three points $x,y,z\in S$, we abusively denote by $\Delta(x,y,z)$ a triangle with vertices $x,y,z$, with no mention to which geodesic segments are chosen for the sides (geodesics between points are not necessarily unique, as seen for example on a sphere, between any two opposite points). The perimeter of a triangle $\Delta=\Delta(x,y,z)$ is defined as $\per(\Delta)=d(x,y)+d(y,z)+d(x,z)$. It does not depend on the choice of the sides. 

\begin{definition}
	Let $\kappa\in\R$ and $\Delta$ be a triangle in $M$ with $\per(\Delta)<2D_\kappa$. A comparison triangle for $\Delta$ in the model space $M_\kappa$ is a triangle $\bar\Delta\subseteq M_\kappa$ with same side lengths as $\Delta$, i.e., if $\Delta=\Delta(x,y,z)$, then $\bar\Delta=\Delta(\bar x,\bar y,\bar z)$ where $\bar x,\bar y,\bar z$ are any points in $M_\kappa$ satisfying 
	$$\begin{cases} d(x,y)=d_\kappa(\bar x,\bar y) \\ d(y,z)=d_\kappa(\bar y,\bar z) \\ d(x,z)=d_\kappa(\bar x,\bar z). \end{cases}$$
\end{definition}

Note that a comparison triangle in $M_\kappa$ is always unique up to rigid motions. We are now ready to define curvature bounds. Intuitively, we say that $(M,d)$ has global curvature bounded from above (resp. below) by $\kappa$ if all its triangles (with perimeter smaller than $2D_\kappa$) are thinner (resp. fatter) than their comparison triangles in the model space $M_\kappa$. 

\begin{definition}
	Let $\kappa\in\R$. We say that $(M,d)$ has global curvature bounded from above (resp. below) by $\kappa$ if and only if for all triangles $\Delta\subseteq M$ with $\per(\Delta)<2D_\kappa$ and for all $x,y\in \Delta$, $d(x,y)\leq d_\kappa(\bar x,\bar y)$ (resp. $d(x,y)\leq d_\kappa(\bar x,\bar y)$), where $\bar x$ and $\bar y$ are the points on a comparison triangle $\bar\Delta$ in $M_\kappa$ that correspond to $x$ and $y$. We then call $(M,d)$ a $\CAT(\kappa)$ (resp. $\CAT^+(\kappa)$) space. 
\end{definition}

\begin{lemma} \label{lemma:convCAT} \cite[Proposition 3.1]{Ohtaconvexity}

Let $(M,d)$ be $\CAT(\kappa)$ for some $\kappa\in\R$ and assume that $M$ has diameter $D<D_\kappa/2$. Then, for all $x_0\in M$, the function $d(\cdot,x_0)$ is convex on $M$. Moreover, set $\alpha_{D,\kappa}=\frac{\sqrt\kappa D}{\tan(\sqrt\kappa D)}$ if $\kappa>0$, $\alpha_{D,\kappa}=1$ otherwise. Then, the function $\frac{1}{2}d^2(\cdot,x_0)$ is $\alpha_{D,\kappa}$-strongly convex, i.e., it satisfies 
$$\frac{1}{2}d^2(\gamma(t),x_0) \leq \frac{1-t}{2}d^2(x,x_0)+\frac{t}{2}d^2(y,x_0)-\frac{\alpha_{D,\kappa} t(1-t)}{4}d(x,y)^2,$$
for all $x,y\in M$ and $\gamma\in\Gamma_{x,y}$. 

\end{lemma}

\subsection{Examples of CAT spaces}

In this section, we give a few examples of CAT spaces. First, the following lemma, which is a direct consequence of \cite[Theorem 1A.6]{bridson2013metric} together with \cite[Theorem 9.2.9]{burago2001course}, give a whole class of Riemannian manifolds as CAT spaces.

\begin{lemma} \label{lemma:RiemCAT}
	Let $(M,g)$ be a simply connected Riemannian manifold. Assume that all the sectional curvatures are bounded from above by some $\kappa\in \R$. Then, $M$ is $\CAT(\kappa)$.
\end{lemma}

We give two examples that are of particular relevance in optimal transport and matrix analysis. 
\begin{enumerate}
	\item Let $M$ be the collection of all $p\times p$ real, symmetric positive definite matrices. Equip $M$ with the metric given by $d(A,B)=\|\log(A^{-1/2}BA^{-1/2})\|_F$, where $\log$ is the matrix logarithm and $\|\cdot\|_F$ is the Frobenius norm. Then, $d$ is inherited from a Riemannian metric and $(M,d)$ is $\CAT(0)$ \cite{bhatia2006riemannian}. This metric space is particularly important in the study of matrix geometric means. For any $A,B\in M$, their geometric mean is defined as $A\# B=A^{1/2}(A^{-1/2}BA^{-1/2})^{1/2}A^{1/2}$, and it is the unique midpoint from $A$ to $B$, i.e., $A\# B=\gamma(1/2)$, where $\gamma$ is the unique geodesic between $A$ and $B$. 
	\item Let $M$ be the collection of all $p\times p$ real, symmetric positive definite matrices with spectrum included in $[\lambda_0,\infty)$, for some $\lambda_0>0$. Equip $M$ with the Bures-Wasserstein metric, obtained as follows. For $A,B\in M$, let $d(A,B)=\textrm{W}_2(\mathcal N_p(0,A),\mathcal N_p(0,B))$, where $\textrm W_2$ is the Wasserstein distance and $\mathcal N_p(0,A)$ (resp. $\mathcal N_p(0,B)$) is the $p$-variate Gaussian distribution with mean $0$ and covariance matrix $A$ (resp. $B$). Then, $d$ is also inherited from a Riemannian metric \cite{bhatia2019bures} and $(M,d)$ is $\CAT(3/(2\lambda_0))$, by \cite[Proposition 2]{massart2019curvature}. 	
\end{enumerate}

Further examples are given below. 

\begin{itemize}
	\item Metric trees are $\CAT(\kappa)$ for any $\kappa\in\R$ (since all its triangles are flat) \cite[Section II.1.15]{bridson2013metric}
	\item The surface of a smooth convex body in a Euclidean space is $\CAT(\kappa)$, where $\kappa>0$ is an upper bound on the reach of the complement of the body, as a sconsequence of \cite[Theorem 3.2.9]{schneider2014convex}
	\item The space of phylogenetic trees \cite{billera2001geometry} is $\CAT(0)$. 
	
\end{itemize}

\end{document}